\documentclass[11pt]{article}
\usepackage[utf8]{inputenc}
\usepackage{amsmath}
\usepackage{amsfonts}
\usepackage{amssymb}
\usepackage{amsthm}
\usepackage{amsrefs}
\usepackage{mathrsfs}
\usepackage[all]{xy}
\usepackage{enumerate}

\theoremstyle{plain}
  \newtheorem{thm}{Theorem}[section]
  \newtheorem{lem}[thm]{Lemma}
  \newtheorem{prop}[thm]{Proposition}
  \newtheorem{cor}[thm]{Corollary} 
\theoremstyle{definition}
  \newtheorem{defn}[thm]{Definition}
  \newtheorem{rmk}[thm]{Remark}
  \newtheorem{ex}[thm]{Example}

\theoremstyle{plain}

\numberwithin{equation}{section}
\allowdisplaybreaks

\def\om{\omega}
\def\Om{\Omega}
\def\Ga{\Gamma}
\def\dpp{\partial_+}
\def\dpm{\partial_-}
\def\dpa{{\partial_{+A}\,}}
\def\dma{{\partial_{-A}\,}}
\def\w{\wedge}

\def\tr{\mathrm{tr}}

\newcommand{\overbar}[1]{\mkern 2.5mu\overline{\mkern-2.5mu#1\mkern 1.5mu}\mkern-1.5mu}

\begin{document}

\title{
\bf\
{Symplectically Flat Connections and Their Functionals}
}

\author{Li-Sheng Tseng and Jiawei Zhou \\
\\
}

\date{October 6, 2022}
\maketitle

\begin{abstract} 

We continue our study of symplectically flat bundles.  We broaden the notion of symplectically flat connections on symplectic manifolds to $\zeta$-flat connections on smooth manifolds.  These connections on principal bundles can be represented by maps from an extension of the base's fundamental group to the structure group.  We introduce functionals with zeroes being symplectically flat connections and study their critical points. Such functionals lead to novel geometric flows.  We also describe some characteristic classes of the $\zeta$-flat bundles.
\end{abstract}

\tableofcontents

\section{Introduction}

In our previous work \cite{TZ}, we introduced the notion of symplectically flat connections on a symplectic manifold $(M^{2n}, \om)$.  For a fiber bundle $\pi:E\to M$ with covariant derivative $d_A$, and $A$ the corresponding connection on $E$, symplectically flat connections have a curvature $F$ that satisfies
\begin{align}\label{defsf}
F& =\Phi\,\omega \,,\qquad
d_A\Phi=0\,,
\end{align}
where $\Phi\in \Omega^0(M, \mathrm{End}\,E)$.  
When the dimension $d=2n\geq 4$, the second equation is implied by the first equation and the Bianchi identity.  When $d=2$, the first equation is trivial and $\Phi$ being covariantly constant becomes the sole condition. 

Let us recall two properties of symplectically flat connections pointed out in \cite{TZ}.  First, with respect to a compatible metric, a symplectically flat connection satisfies the Yang-Mills equations \cite{TZ}*{Example 3.3}. (In this paper, we will always assume that any Riemannian metric we consider on a symplectic manifold is compatible to the symplectic structure.) 
In fact, in dimension $d=2$, the notion of Yang-Mills connections with respect to a compatible metric is equivalent to symplectically flat connections.  This is significant since Atiyah and Bott in a seminal work \cite{Atiyah-Bott} gave the classification of Yang-Mills bundles over Riemann surfaces.  Their result thus immediately provides us with the classification of the symplectically flat bundles over closed, oriented two-manifolds.  

In dimensions four and higher, it should be evident from \eqref{defsf} that the symplectic flatness condition is a much stronger condition than the Yang-Mills condition.  Perhaps a better comparison for symplectic flatness in this case is that of the standard flat connection condition, $F=0$.  This relates to a second property of symplectically flatness described in \cite{TZ} that we now recall.  If the symplectic structure is integral class, i.e. $\om \in H^2(M, \mathbb{Z})$, then we can build what is called a prequantum circle bundle $X$, i.e. $S^1 \rightarrow X \rightarrow M$, which by definition has Euler class given by $\om$.  Of interest, a symplectically flat connection $A$ over $M$ satisfying \eqref{defsf} can be lifted to a flat $A-\theta\Phi$ connection over $X$, where $\theta$ is a global angular form of $X$ \cite{TZ}*{Corollary 3.9}.   And as is well-known, isomorphism classes of flat bundles over any manifold is classified by the $G$-character variety of $\pi_1$, i.e. the conjugacy classes of homomorphisms from the fundamental group of the base space to the structure group $G$ of the fiber.    

Together, we have for dimension $d=2$ a classification of symplectically flat bundles from Atiyah-Bott, and for $d\geq 4$, we have in the integral $\om$ case a suggestion of a possible classification related to flat bundles over a space of one higher dimension.  A natural question then is to ask whether we can obtain a classification of bundles with a symplectically flat connection in dimensions $d=2n>2$ over any symplectic manifold $(M^{2n}, \om)$, which generalizes the classification of Atiyah-Bott's in dimension two.

Indeed, we are able to obtain a classification of symplectically flat bundles that holds for all dimensions.  Our result is actually a corollary of a general statement, that classifies principal bundles with curvature proportional to any closed two-form.  For this, it is useful to define a broader notion of flatness.   
\begin{defn}\label{zfdef}
For a smooth manifold $M$, let $\zeta\in\Omega^2(M)$ be a closed 2-form on $M$.  We say a fiber bundle $E$ over $M$ with covariant derivative $d_A$, where locally  $d_A=d+A$ with $A$ being the connection 1-form on $E$, is {\bf $\zeta$-flat} if the curvature $F$ satisfies  
\begin{align}\label{zflat}
F&=\Phi\, \zeta\,,\qquad 
d_A\Phi =0\,.
\end{align}
We shall call such a connection $A$ with curvature satisfying \eqref{zflat} a {\bf $\zeta$-flat connection}. 
\end{defn}
This definition does not require $\zeta$ to be a symplectic form and is also applicable for bundles over odd dimensional base manifolds.  For this general $\zeta$-flat condition, we prove in Section 2 the following:
\begin{thm}\label{Thrm1}
Let $M$ be a connected manifold, $\zeta\in\Omega^2(M)$ be a closed 2-form, and $G$ be a Lie group. There exists a bijective correspondence between the following sets:
\begin{align*}
\left\{\begin{matrix}\text{isomorphism classes of G-bundles over M}\\ \text{with a $\zeta$-flat connection}\end{matrix}\right\}
\simeq
\left\{\begin{matrix}\text{conjugacy classes of}\\ \text{homomorphisms }
\rho: \Gamma\to G\end{matrix}\right\}.
\end{align*}
Here, $\Gamma$ is an $\mathbb{R}/\overbar{H}$-extension of $\pi_1(M)$, where $\overbar{H}\subset\mathbb{R}$ is the closure of a group $H$ defined as 
$$
H=\left\{ \int_D\zeta \ \big| \  D \text{ is a representative in } \pi_2(M) \right\}.
$$
\end{thm}
When $\zeta=\om$ is the symplectic structure, the above theorem gives a classification of symplectically flat bundles.  In $d=2$, this reproduces exactly the classification of Yang-Mills bundles over Riemann surfaces of Atiyah-Bott.  In the case when $\overbar{H}=\mathbb{R}$, then $\Ga=\pi_1(M)$ and the classification reduces to exactly that for flat connections.  And just like the classification for flat connections, Theorem \ref{Thrm1} holds for base manifolds of even or odd dimensions. 

That symplectically flat connections are Yang-Mills connections and also related to flat connections make them rather special.  Recall that flat connections have a special role in Yang-Mills theory, which is described by the functional
\begin{align}
S_{Y\!M}(A) = \|F(A)\|^2=\int_M \tr~ F \w * F \,.
\end{align}
Flat connections have the distinguished role that they correspond to the zeroes of the Yang-Mills functional.  From this perspective, we ask in Section 3 whether there are functionals of connection 1-forms whose zeroes are exactly symplectically-flat connections, or more generally $\zeta$-flat connections.  For this, we find two natural functionals.

\begin{defn}\label{dpYM}
The \textbf{primitive Yang-Mills} functional is defined on a connection $A$ as
\begin{align*}
S_{pY\!M}(A)= \|F_p(A)\|^2= \int_M \tr~ F_p \w * F_p 
\end{align*}
where under the Lefschetz decomposition of the curvature 2-form, $F= F_p + \Phi\, \om$, $F_p$ is the primitive component of $F$. We call the critical points of $S_{pY\!M}$ \textbf{primitive Yang-Mills connections}.
\end{defn}

For the second functional, it can be easily motivated starting from the sum of the norm-squared of the two conditions of symplectic flatness \eqref{defsf}: $\|F(A)-\omega\,\Phi\|^2+\|d_A\Phi\|^2$.  If we let $\Phi$ to be independent of $A$, or explicitly, replacing $\Phi$ with a general endomorphism on the bundle denoted by $B$, this leads us to the following functional.  

\begin{defn}\label{def of cYM}
The \textbf{cone Yang-Mills} functional is defined on a connection $A$ and an endomorphism $B$ on the bundle as 
\begin{align*}
S_{cY\!M}(A,B)= \|F(A)+\om B\|^2 + \|d_A B\|^2= \int_M \tr \left[(F + \om B) \w * (F + \om B) + d_A B \w * d_A B\right]
\end{align*}
\end{defn}
Clearly, the zeroes of both functionals satisfy \eqref{defsf} and correspond to symplectically flat connections.  But they interestingly have different systems of equations for the critical points with solutions generally different from the Yang-Mills connections.  These functionals can also be modified so that the zeroes are $\zeta$-flat instead.

It is worthwhile to point out that the second functional also has an interpretation as the normed square of the cone curvature.  The cone algebra of interest here consists of elements $\mathcal{C}^*(M):=\Om^*(M)[\theta]=\{ \xi+\theta\eta\, |\, \xi,\eta\in\Omega^*(M) \}$ with $d\theta=\om$ (see \cite{TT} for a discussion of this cone algebra in the context of symplectic geometry).  This algebra is quasi-isomorphic to the algebra of differential forms of the prequantum circle bundle $X$, $\Om^*(X)$, when $\om$ is integral class.  But importantly, $\mathcal{C}^*(M)$ is well-defined even when $\om$ is not integral.  As discussed in \cite{TZ}, we can define a covariant derivative on the twisted $\mathcal{C}^*(M,E)$ given by $D_C = d_A + \theta B$, with the cone curvature $D_{\mathcal{C}}^2= \left(F + \om\, B\right) - \theta \left(d_A B\right)$.  As noted in \cite{TZ}*{Proposition 3.8}, the cone curvature is flat, i.e. $D_{\mathcal{C}}^2=0$, if and only if $A$ is a symplectically flat connection.  This justifies calling $S_{cYM}$ the cone Yang-Mills functional.

The equivalences between symplectically flat and cone flat, and also their lift to flat connections on the prequantum circle bundle when $\om$ is integral do provide us with further insights.  Though flat connections have zero curvature, there are known characteristic classes associated to flat bundles.  (See for example \cite{Morita}*{Chapter 2.3} for an overview.)  This led us to consider characteristic classes of flat bundles over $X$ but with the covariant derivative taken to be that of the cone.  Specifically, using $D_C = d_A - \theta \Phi$ with $F=\Phi \, \om$, we are able to obtain characteristic classes for the symplectically flat bundles.  Moreover, the Chern-Simons type functionals on $X$ also lead to novel functionals for symplectically flat connections.

This paper is organized as follows. In Section 2, we give the proof of Theorem \ref{Thrm1}, the  classification theorem of $\zeta$-flat bundles.  In Section 3, we introduce the primitive Yang-Mills and cone Yang-Mills functionals and describe some of their properties.  We provide explicit examples of connections that are either Yang-Mills or primitive Yang-Mills, but not both.  We also write down functionals whose zeroes are more generally $\zeta$-flat.  We conclude in Section 4 with a brief discussion on the gradient flow of the primitive Yang-Mills functional and some remarks concerning the characteristic forms and classes for symplectically flat and $\zeta$-flat bundles. 

\

\noindent{\it Acknowledgements.~} 
We would like to thank David Clausen, Si Li, Jianfeng Lin, Xiang Tang, Jiaping Wang,  Shing-Tung Yau and Xiangwen Zhang for helpful discussions.  We would like to acknowledge the support of the Simons Collaboration Grant No.~636284 for the first author and the support of the National Key Research and Development Program of China No.~2020YFA0713000 for the second author.

\section{Classification of  $\zeta$-flat bundles}

The definition of symplectically flat connection is given in \eqref{defsf}. A natural question to ask is how symplectically flat bundles can be classified.  As we have seen, when the base is two-dimensional, symplectically flat bundles are Yang-Mills bundles, which are classified by Atiyah and Bott \cite{Atiyah-Bott}. Hence, we seek a classification when the base space is of dimensions greater than two.  In this case, we find a  classification of $\zeta$-flat bundles (see Definition \ref{zfdef}) that generalizes Atiyah-Bott's result for Yang-Mills bundles over Riemann surfaces. Our proof will be similar to that given by Morrison \cite{Morrison} for Yang-Mills bundles in the two-dimensional case.

\begin{thm}
Let $M$ be a  connected manifold, $\zeta\in\Omega^2(M)$ be a closed 2-form, and $G$ be a Lie group. There exists a bijective correspondence between the following sets:
\begin{align*}
\left\{\begin{matrix}\text{isomorphism classes of G-bundles over M}\\ \text{with a $\zeta$-flat connection}\end{matrix}\right\}
\simeq
\left\{\begin{matrix}\text{conjugacy classes of}\\ \text{homomorphisms }
\rho: \Gamma\to G\end{matrix}\right\}.
\end{align*}
Here, $\Gamma$ is an $\mathbb{R}/\overbar{H}$-extension of $\pi_1(M)$, where $\overbar{H}\subset\mathbb{R}$ is the closure of a group $H$ defined as 
$$
H=\left\{ \int_D\zeta \ \big| \  D \text{ is a representative in } \pi_2(M) \right\}.
$$
\end{thm}

To give the exact definition of $\Gamma$, we use the following conventions.

Let $\Psi(M)$ be the path space defined as follows. It consists of equivalence classes of closed, piecewise smooth paths on $M$. If we reparametrize a path $\alpha_1:[0,1]\to M$ by choosing a different map $\alpha_2$ with the same image and orientation, then $\alpha_1$ and $\alpha_2$ are equivalent in $\Psi(M)$. In other words, $\alpha_1$ and $\alpha_2$ are equivalent if there exists a piecewise smooth increasing function $\phi:[0,1]\to [0,1]$ with $\phi(0)=0$, $\phi(1)=1$ and $\alpha_2=\alpha_1\circ\phi$.

We can define a multiplication in $\Psi(M)$ by the composition of paths. For two paths $\alpha_1,\alpha_2:[0,1]\to M$ with $\alpha_1(1)=\alpha_2(0)$, set
$$
(\alpha_2\alpha_1)(t)=
\begin{cases}
\alpha_1(2t), & t\in [0,\frac{1}{2}],\\
\alpha_2(2t-1), & t \in [\frac{1}{2},1].
\end{cases}
$$

Fix a base point $a\in M$. Let $\Psi^a_a(M)$ be the semigroup in $\Psi(M)$ that consists of classes of loops with base point $a$, $\Psi^a_{a,0}(M)$ be the subsemigroup of $\Psi^a_a(M)$ that consists of classes of contractible loops, and $\Psi^a_{a,\zeta}(M)$ be the subsemigroup of $\Psi^a_{a,0}(M)$ that consists of classes of loops which are the boundary of some disk (contractible 2-chain) $D$ in $M$ satisfying $\int_D\zeta=0$. Then $\Psi^a_a(M)/\Psi^a_{a,0}(M)$ and $\Psi^a_{a,0}(M)/\Psi^a_{a,\zeta}(M)$ have group structures, where the identity element is the equivalence class of the constant path at $a$, and the inverse is reversing the orientation:
$$
\alpha^{-1}(t)=\alpha(1-t).
$$

Observe that $\Psi^a_a(M)/\Psi^a_{a,0}(M)=\pi_1(M)$, and $\Psi^a_{a,0}(M)/\Psi^a_{a,\zeta}(M)$ is equivalent to a quotient group $\mathbb{R}/H$ by the identification $\gamma\mapsto[\int_D\zeta]$, where $H=\{ \int_D\zeta| D \text{ is a representative in } \pi_2(M) \}$.

When $\mathbb{R}/H$ is a Lie group, we set $\Gamma=\Psi^a_a(M)/\Psi^a_{a,\zeta}(M)$. When $\mathbb{R}/H$ is not a Lie group, we set $\Gamma=\pi_1(M)$.

Explicitly, there are three possibilities for $H$. Let $H^+$ be the subset of $H$ with positive numbers.

\textbf{Case 1.}
When $H^+$ is empty, $H=0$ and $\mathbb{R}/H=\mathbb{R}$. In this case $\Gamma=\Psi^a_a(M)/\Psi^a_{a,\zeta}(M)$ is an $\mathbb{R}$-extension of $\pi_1(M)$. For Riemann surfaces, this occurs when the genus $g\geq 1$.

\textbf{Case 2.}
When $H^+$ has a minimal number, $H\simeq\mathbb{Z}$ and $\mathbb{R}/H \simeq S^1$. In this case $\Gamma=\Psi^a_a(M)/\Psi^a_{a,\zeta}(M)$ is an $S^1$-extension of $\pi_1(M)$. For Riemann surfaces, this occurs when the genus $g=0$.

\textbf{Case 3.}
When $H^+$ is non-empty and has no minimal number, $H$ is dense in $\mathbb{R}$ and $\mathbb{R}/H$ is not a Lie group. In this case, $\Gamma=\pi_1(M)$. This case is impossible when $c\,\zeta$ is an integral cohomology class for some number $c\neq0$.

In the remainder of this subsection, we prove the classification theorem. 

\textbf{Case 1 and 2, $\mathbb{R}/H=\mathbb{R}$ or $S^1$.}

We first consider the case that $\mathbb{R}/H$ is a Lie group. By the discussion above, $\Psi^a_a(M)/\Psi^a_{a,\zeta}(M)$ is either an $\mathbb{R}$ or an $S^1$ extension of $\pi_1(M)$. 

\noindent\textbf{Step 1.} Given a morphism $\rho:\Psi^a_a(M)/\Psi^a_{a,\zeta}(M)\to G$, construct a corresponding principal $G$-bundle $P_{\rho}$ with a connection $A_{\rho}$.

Let $\Psi^a(M)$ be the space of classes of paths starting at $a$ on $M$. Two paths $\delta_1,\delta_2$ are identified if $\delta_1(1)=\delta_2(1)$ and $\delta_2^{-1}\delta_1\in\Psi^a_{a,\zeta(M)}$. This $\Psi^a(M)$ is a principal bundle over $M$. Its structure group is $\Psi^a_a(M)/\Psi^a_{a,\zeta}(M)$, and the projection is $\tau:\Psi^a(M)\to M,\delta\mapsto\delta(1)$. An element $\gamma\in\Psi^a_a(M)/\Psi^a_{a,\zeta}(M)$ acts on $\delta\in \Psi^a(M)$ on the right by $\delta\mapsto\delta\gamma$. 

Let $P_\rho=\Psi^a(M) \times_\rho G$ be the associated $G$-bundle by identifying $(\delta,g)$ and $\left(\delta\gamma,\rho(\gamma^{-1})g\right)$ in $\Psi^a(M)\times G$ for each $\gamma\in\Psi^a_a(M)/\Psi^a_{a,\zeta}(M)$. $G$ acts on $P_\rho$ by $[\delta,g]h=[\delta,gh]$. To define a connection $A_\rho$ on $P$, we can construct horizontal lifts for any path $\alpha$ on $M$ as follows. For arbitrary $[\delta,g]\in (P_\rho)_{\alpha(0)}$ with $\delta(1)=\alpha(0)$, we define the horizontal lift of $\alpha$ starting from $[\delta,g]\in P_\rho$ by $\tilde{\alpha}(t)=[\alpha_t\delta,g]$. Here, $\alpha_t$ is part of $\alpha$ starting at $\alpha(0)$ and ending at $\alpha(t)$. i.e. $\alpha_t:[0,1]\to M,s\mapsto\alpha(st)$.

We verify that this connection is well defined. Let $\alpha$ and $\beta$ be two paths on $M$ with $\alpha'(0)=\beta'(0)$. By definition, we have $\tilde{\alpha}'(0)=\tilde{\beta}'(0)$. Thus, the horizontal lifts do induce a distribution for horizontal vectors. Also, suppose $\alpha$ is an arbitrary path on $M$. For $[\delta_{(1)},g_1],[\delta_{(2)},g_2]\in (P_\rho)_{\alpha(0)}$, there exists an $h\in G$ such that $[\delta_{(1)},g_1]=[\delta_{(2)},g_2]h$. Let $\tilde{\alpha}_{(1)}$ and $\tilde{\alpha}_{(2)}$ be the horizontal lifts of $\alpha$ from $[\delta_{(1)},g_1]$ and $[\delta_{(2)},g_2]$, respectively. By assumption, we have $[\delta_{(1)},g_1]=\left[\delta_{(2)}\rho(g_1h^{-1}g_2^{-1}),g_1\right]$, which implies ${\delta_{(1)}}^{-1}\delta_{(2)}\rho\left(g_1h^{-1}g_2^{-1}\right) \in \Psi^a_{a,\zeta}(M)$. Hence, $\big(\alpha_t\delta_{(1)}\big)^{-1}\alpha_t\delta_{(2)}\rho\left(g_1h^{-1}g_2^{-1}\right) \in \Psi^a_{a,\zeta}(M)$ for all $t\in [0,1]$. So the lifts satisfy
$$
\tilde{\alpha}_{(1)}(t)=[\alpha_t\delta_{(1)},g_1]=[\alpha_t\delta_{(2)},g_2]h=\tilde{\alpha}_{(2)}(t)h.
$$
This proves that the horizontal lift does not depend on the choice of representative of $[\delta,g]$, and is equivariant along the fiber.

\noindent\textbf{Step 2.} Verify that the connection $A_{\rho}$ is $\zeta$-flat.

\begin{lem}
The connection $A_\rho$ constructed above is $\zeta$-flat.
\end{lem}

\begin{proof}
Let $p\in M$ be an arbitrary point and $v_1,v_2\in T_pM$ be arbitrary linearly independent vectors. Suppose $\zeta(v_1,v_2)=c$. We can find a local coordinate $\{x_1,\cdots,x_N\}$ such that $v_1=\frac{\partial}{\partial x_1}$, $v_2=\frac{\partial}{\partial x_2}$, and $\zeta=c\,dx_1\wedge dx_2+\bar{\zeta}$, where $\bar{\zeta}$ is generated by other $dx_i\wedge dx_j$ locally except for $dx_1\wedge dx_2$. Let $D_{(t)}$ be the parallelogram spanned by $\sqrt{t}v_1$ and $\sqrt{t}v_2$ in the local coordinate system, and $\gamma_{(t)}=\partial D_{(t)}$ be its boundary. Then we have $\int_{D_{(t)}} \zeta=t\,c$.

At arbitrary $[\delta,g]\in P_\rho$ on the fiber of $p$, let $v_1^H$ and $v_2^H$ be the horizontal lift of $v_1$ and $v_2$, respectively. Then the curvature
$$
F\left(v_1^H,v_2^H\right)=\frac{\partial}{\partial t}hol\left(\gamma_{(t)}\right).
$$
Here, $hol\left(\gamma_{(t)}\right)\in G$ denotes the holonomy along $\gamma_{(t)}$ at $[\delta,g]$.

On the other hand,
$$
[\delta,g]hol\left(\gamma_{(t)}\right)=[\gamma_{(t)}\delta,g]=[\delta(\delta^{-1}\gamma_{(t)}\delta),g]=\left[\delta,g\left(g^{-1}\rho\left(\delta^{-1}\gamma_{(t)}\delta\right)g\right)\right].
$$
This implies
$$
hol\left(\gamma_{(t)}\right)=g^{-1}\rho\left(\delta^{-1}\gamma_{(t)}\delta\right)g.
$$
Note that $\delta^{-1}\gamma_{(t)}\delta$ is contractible and its base point is $a$, i.e. $\delta^{-1}\gamma_{(t)}\delta \in \Psi^a_{a,0}(M)$. Since $\Psi^a_{a,0}(M)/\Psi^a_{a,\zeta}(M)$ is equivalent to $\mathbb{R}/H$, $\delta^{-1}\gamma_{(t)}\delta$ can be identified with $\int_{D_{(t)}}\zeta=tc$ in $\mathbb{R}/H$. Let $\xi$ be an element of the Lie algebra of $\Psi^a_a(M)/\Psi^a_{a,\zeta}(M)$ such that $\mathrm{exp}(t\xi)$ generates the subgroup $\Psi^a_{a,0}(M)/\Psi^a_{a,\zeta}(M)$ and $t$ can be identified with $\mathbb{R}/H$. This $\xi$ is independent of the choice of $v_1$ and $v_2$. Then we have
$$
F\left(v_1^H,v_2^H\right)=\frac{\partial}{\partial t}hol\left(\gamma_{(t)}\right)=c\mathrm{Ad}_{g^{-1}}\rho(\xi)=\zeta(v_1,v_2)\mathrm{Ad}_{g^{-1}}\rho(\xi).
$$
\end{proof}

\noindent\textbf{Step 3.} Show that the morphism $\rho\mapsto(P_{\rho},A_{\rho})$ is invariant under conjugation.

\begin{lem}
Suppose $\rho$ and $\bar{\rho}$ are conjugate homomorphisms from $\Psi^a_a(M)/\Psi^a_{a,\zeta}(M)$ to $G$. $(P_\rho,A_\rho)$ and $(P_{\bar{\rho}},A_{\bar{\rho}})$ are principal bundles with $\zeta$-flat connections constructed as above. Then $(P_\rho,A_\rho)$ and $(P_{\bar{\rho}},A_{\bar{\rho}})$ are equivalent.
\end{lem}
\begin{proof}
Suppose $\bar{\rho}=g_0\rho g_0^{-1}$, we define an automorphism on $\Psi^a(M) \times G$ by $(\delta,g)\mapsto(\delta,g_0g)$ and this automorphism induces an isomorphism $f:\Psi^a(M) \times_\rho G\to \Psi^a(M) \times_{\bar\rho} G$. For arbitrary path $\alpha$ on $M$ and $[\delta,g]_\rho\in \Psi^a(M)$ such that $\delta(1)=\alpha(0)$, the horizontal lift of $\alpha$ at $[\delta,g]_\rho$ is defined as $\tilde{\alpha}_\rho(t)=[\alpha_t\delta,g]_\rho$. On the other hand, the horizontal lift of $\alpha$ at $f\big( [\delta,g]_\rho \big)=[\delta,g_0g]_{\bar{\rho}}$ is defined as $\tilde{\alpha}_{\bar{\rho}}(t)=[\alpha_t\delta,g_0g]_{\bar{\rho}}$. So we have $\tilde{\alpha}_{\bar{\rho}}(t)=f\circ \tilde{\alpha}_\rho(t)$, which implies $f_*$ sends horizontal vectors to horizontal vectors. Thus, $f^*A_{\bar\rho}=A_\rho$.
\end{proof}

By this lemma, given any conjugacy classes $[\rho]$ of the morphisms from $\Psi^a_a(M)/\Psi^a_{a,\zeta}(M)$ to $G$, there is a corresponding $G$-bundle $P_{\rho}$ with a $\zeta$-flat connection $A_{\rho}$. It remains to show that this correspondence is bijective.

\noindent\textbf{Step 4.} The morphism $[\rho]\mapsto(P_{\rho},A_{\rho})$ is injective.

\begin{lem}
Let $\rho,\bar{\rho}:\Psi^a_a(M)/\Psi^a_{a,\zeta}(M)\to G$. $\rho$ and $\bar{\rho}$ are conjugate when $(P_\rho,A_\rho)$ and $(P_{\bar{\rho}},A_{\bar{\rho}})$ are equivalent.
\end{lem}

\begin{proof}
Suppose $f:P_\rho \to P_{\bar\rho}$ is a $G$-bundle isomorphism and $f^*A_{\bar\rho}=A_\rho$. Set $h:\Psi^a(M)\to G$ such that
$$
f\big( [\delta,e]_\rho \big)=[\delta,e]_{\bar\rho}h(\delta)
$$
where $\delta\in \Psi^a(M)$ and $e$ is the identity of $G$. For each $\gamma\in \Psi^a_a(M)/\Psi^a_{a,\zeta}(M)$, we have
\begin{align*}
 [\delta,e]_{\bar\rho}h(\delta)\rho(\gamma)&=f\big( [\delta,e]_\rho \big)\rho(\gamma)=f\big( [\delta,e]_\rho\rho(\gamma) \big)\\
&= f\big( [\delta\gamma,e]_\rho \big)=[\delta\gamma,e]_{\bar\rho} h(\delta\gamma)=[\delta,e]_{\bar\rho}\bar{\rho}(\gamma)h(\delta\gamma).
\end{align*}
This implies $\rho(\gamma)=h(\delta)^{-1}\bar{\rho}(\gamma)h(\delta\gamma)$. We will prove that $h$ is a constant.

For arbitrary $\delta\in \Psi^a(M)$, let $\delta_t$ denote the part of the path $\delta$ starting at $\delta(0)=a$ and ending at $\delta(t)$. i.e. $\delta_t:[0,1]\to M,s\mapsto\delta(st)$. By the construction of the connections, $[\delta_t,e]_\rho$ and $[\delta_t,e]_{\bar{\rho}}$ are horizontal paths on $P_\rho$ and $P_{\bar{\rho}}\,$, respectively. The projection of these two paths on $M$ are exactly $\delta$. By the definition of $h$, we have
$$
f\big( [\delta_t,e]_\rho \big)=[\delta_t,e]_{\bar\rho}h(\delta_t).
$$
On the other hand, since $f^*A_{\bar\rho}=A_\rho$, $f\big( [\delta_t,e]_\rho \big)$ is a horizontal lift of $\delta$ on $P_{\bar{\rho}}$ passing through $f\big( [\delta_0,e]_\rho \big)$. As horizontal curves are equivariant by right $G$-action, $[\delta_t,e]_{\bar\rho}h(\delta_0)$ is also a horizontal lift of $\delta$ on $P_{\bar{\rho}}$ passing through $[\delta_0,e]_{\bar\rho}h(\delta_0)=f\big( [\delta_0,e]_\rho \big)$. Therefore, we have
$$
f\big( [\delta_t,e]_\rho \big)=[\delta_t,e]_{\bar\rho}h(\delta_0).
$$
Comparing the two equations, we have $h(\delta_t)=h(\delta_0)$ for all $t\in [0,1]$. This implies $h(\delta)$ is equal to $h$ acting on the constant path at $a$. So $h$ is a constant. Therefore, we have $\rho$ and $\bar{\rho}$ are conjugate.
\end{proof}

\noindent\textbf{Step 5.} The morphism $[\rho]\mapsto(P_{\rho},A_{\rho})$ is surjective.

The following lemma leads to surjectivity. It also holds when $\mathbb{R}/H$ is not a Lie group, which we will discuss later.
\begin{lem}\label{holonomy}
Suppose $\pi:P\to M$ be a principal $G$-bundle, and $A$ is a $\zeta$-flat connection on $P$ with curvature $F=\Phi\zeta$. There exists $\xi\in\mathfrak{g}$ such that for any contractible loop $\gamma$ starting at $a$ and any oriented disk $D$ in $M$ with $\partial D=\gamma$ ($\partial D$ and $\gamma$ also have the same orientation), the holonomy along $\gamma$ is
$$
hol(\gamma)=-\mathrm{exp}\Big( \int_D \zeta \cdot \xi \Big).
$$
\end{lem}
\begin{proof}
Take a point $u_0\in P$ on the fiber of $a$. Consider the holonomy bundle
$$
\hat{P}=\{ u\in P|\text{there exists horizontal path } \delta \text{ such that } \delta(0)=u_0,\delta(1)=u \}.
$$
$\hat{P}$ is a principal bundle over $M$, and its structure group $\hat{G}$ is the holonomy group of $P$ at $u_0$. Let $\hat{A}$ be the restriction of $A$ on $\hat{P}$. By the holonomy theorem of Ambrose and Singer \cite{AS}, the Lie algebra of $\hat{G}$ is
$$
\hat{\mathfrak{g}}=\mathrm{span}\left\{ F\left(v_1^H,v_2^H\right)\big| v_1^H,v_2^H \text{ are horizontal vectors at }u\text{ for some }u\in\hat{P} \right\}.
$$
By assumption, $F=\Phi\zeta$, or more precisely $F=\Phi(\pi^*\zeta)$. Since $\Phi$ is covariantly constant, it is equal to some $\xi\in\hat{\mathfrak{g}}$ at any point in $\hat{P}$. Hence, $F\left(v_1^H,v_2^H\right)\in \mathbb{R}\xi$. So $\hat{\mathfrak{g}}$ is 1-dimensional and abelian. Then $\hat{A}=\xi\otimes\theta$ for some $\theta\in \Omega^1(M,Ad\,\hat{P})$ and $d\theta=\hat{\pi}^*\zeta$, where  $\hat{\pi}:\hat{P}\to M$ is the projection.

For an arbitrary contractible loop $\gamma$ and a disk $D$ such that $\partial D=\gamma$, there exists a contractible neighborhood $U\subset M$ of $D$. Then $\hat{P}|_U=U\times \hat{G}$ is trivial. Let $\sigma:U\to \hat{P}|_U$ be a local section and $\psi:\hat{P}|_U\to U\times \hat{G}$ be a trivialization such that $\psi\circ\sigma(p)=(p,e)$ for $p\in M$. Then $(\psi^{-1})^*\hat{A}=(\xi\otimes\sigma^*\theta,0)+(0,MC_{\hat{G}})$, where $MC_{\hat{G}}:T\hat{G}\to\hat{\mathfrak{g}}$ is the Maurer-Cartan form of $\hat{G}$ sending a vector to the corresponding invariant vector field. Observe that $d(\sigma^*\theta)=\zeta$. The horizontal lift $\tilde{\gamma}$ of $\gamma$ with $\tilde{\gamma}(0)=\sigma(a)$ satisfies $\psi\circ\tilde{\gamma}(t)=(\gamma(t),g(t))$ with $g(t)\in \hat{G}$ and $g(0)=e$. Then
\begin{align*}
0 = \hat{A}\big(\tilde{\gamma}'(t)\big) \Big|_{\gamma(t_0)} 
&= (\psi^{-1*}\hat{A})\big(\gamma'(t),0\big) \Big|_{\big(\gamma(t_0),g(t_0)\big)}+(\psi^{-1*}\hat{A})\big(0,g'(t)\big) \Big|_{\big(\gamma(t_0),g(t_0)\big)} \\
&= (\sigma^*\theta)\big(\gamma'(t)\big)\cdot \xi \Big|_{\big(\gamma(t_0) \big)}+MC_{\hat{G}}\big(g'(t)\big) \Big|_{\big( g(t_0)\big)}.
\end{align*}
So we have
$$
g(t_0)=-\mathrm{exp}\Big( \int_0^{t_0}(\sigma^*\theta)\big(\gamma'(t)\big)dt\cdot \xi \Big),
$$
and the holonomy along $\gamma$ is
$$
hol(\gamma)=g(1)=-\mathrm{exp}\Big( \int_{\gamma}\sigma^*\theta \cdot \xi \Big)=-\mathrm{exp}\Big( \int_D \zeta \cdot \xi \Big).
$$
\end{proof}

Now we prove the surjectivity.

\begin{lem}
Let $\pi:P\to M$ be a principal $G$-bundle with a $\zeta$-flat connection $A$. There exists a morphism $\rho:\Psi^a_a(M)/\Psi^a_{a,\zeta}(M)\to G$ such that $(P_\rho,A_\rho)$ and $(P,A)$ are equivalent.
\end{lem}
\begin{proof}
For each $\gamma\in\Psi^a_{a,\zeta}(M)$, there exists some disk $D$ such that $\partial D=\gamma$, and $\int_D \zeta=0$. By Lemma $\ref{holonomy}$, the holonomy along $\gamma$ is $hol(\gamma)=-\mathrm{exp}\Big( \int_D \zeta \cdot \xi \Big)=e$. So we can define a morphism 
$$
\rho:\Psi^a_a(M)/\Psi^a_{a,\zeta}(M)\to G,\quad \gamma\mapsto hol(\gamma).
$$
We will show that $(P_\rho,A_\rho)$ and $(P,A)$ are equivalent.

Given $[\delta,g]\in P_\rho$, let $\tilde{\delta}$ be the horizontal lift of $\delta$ in $P$ with $\tilde{\delta}(0)=u_0$. Set
$$
f:P_\rho\to P,\quad [\delta,g]\mapsto \tilde{\delta}(1)g.
$$
Suppose $[\delta_{(1)},g_1]$ and $[\delta_{(2)},g_2]$ represent the same class in $P_\rho$, then there exists $\gamma_0\in \Psi^a_a(M)/\Psi^a_{a,\zeta}(M)$ such that $\delta_{(1)}=\delta_{(2)}\gamma_0$ in $\Psi^a_a(M)/\Psi^a_{a,\zeta}(M)$ and $g_1=\big(\rho(\gamma_0)\big)^{-1}g_2$ in $G$. So $\gamma_0^{-1}\delta_{(2)}^{-1}\delta_{(1)}\in \Psi^a_{a,\zeta}(M)$, and its holonomy is trivial according to Lemma \ref{holonomy}. Thus, the holonomy of $\delta_{(2)}^{-1}\delta_{(1)}$ is $hol(\gamma_0)=\rho(\gamma_0)$. Let $\tilde{\delta}_{(1)}$ and $\tilde{\delta}_{(2)}$ denote the horizontal lift of $\delta_{(1)}$ and $\delta_{(2)}$, respectively, starting at $u_0$. Since $\tilde{\delta}_{(1)}$ and the horizontal lift of $\delta_{(2)}\gamma_0=\delta_{(2)}\delta_{(2)}^{-1}\delta_{(1)}$ starting at $u_0$ have the same ending point, so do $\tilde{\delta}_{(1)}$ and the horizontal lift of $\delta_{(2)}$ starting at $u_0\cdot hol(\gamma_0)$. Then the ending points of $\tilde{\delta}_{(1)}$ and $\tilde{\delta}_{(2)}$ satisfy $\tilde{\delta}_{(1)}(1)=\tilde{\delta}_{(2)}(1)\cdot hol(\gamma_0)$. Therefore, by definition,
$$
f\left([\delta_{(1)},g_1]\right)=\tilde{\delta}_{(1)}(1)g_1=\tilde{\delta}_{(2)}(1)\rho(\gamma_0)g_1=\tilde{\delta}_{(2)}(1)g_2=f(\left[\delta_{(2)},g_2]\right).
$$
Hence, $f$ is well defined. It is straightforward to check that $f$ is a $G$-bundle isomorphism.

Finally, we show that $f^*A=A_\rho$. For an arbitrary path $\alpha$ on $M$ and $[\delta,g]\in P_\rho$ such that $\delta(1)=\alpha(0)$, the horizontal lift of $\alpha$ at $[\delta,g]$ is $\tilde{\alpha}(t)=[\alpha_t\delta,g]$. Here, $\alpha_t$ is defined as before, which is the part of $\alpha$ starting at $\alpha(0)$ and ending at $\alpha(t)$. So $f\left([\alpha_t\delta,g]\right)$ is the ending point of the path on $P$ which is the horizontal lift of $\alpha_t\delta$ starting at $u_0g$. In other words, $f\circ\tilde{\alpha}(t)$ is on the horizontal lift of $\alpha$ starting at $f\left([\delta,g]\right)$ and on the fiber of $\alpha(t)$. Therefore, $f\circ\tilde{\alpha}$ is a horizontal path on $P$. Hence, $f_*$ sends horizontal vectors to horizontal vectors and $f^*(A)=A_\rho$.
\end{proof}

Combining the lemmas above, we have proved that $[\rho]\mapsto(P_{\rho},A_{\rho})$ is an isomorphism.

\textbf{Case 3. $\mathbb{R}/H$ is not a Lie group}

Now we turn to the case that $\mathbb{R}/H$ is not a Lie group. In this case $H^+$ is non-empty and has no minimal number.
\begin{lem} 
When $H^+$ is non-empty and has no minimal number, a $\zeta$-flat connection is flat.
\end{lem}
\begin{proof}
By Theorem \ref{holonomy}, there exists $\xi\in\mathfrak{g}$ such that for any contractible loop $\gamma$ and disk $D$ with $\partial D=\gamma$, we have $hol(\gamma)=-\mathrm{exp}\Big( \int_D \zeta \cdot \xi \Big)$. If $\xi\neq 0$, there exists some small enough $t_0$ such that $\mathrm{exp}(t\xi)\neq e$, the identity element of $G$,  for any $0<t<t_0$.

Since $H^+$ has no minimal number, there exists a closed sphere $\Sigma\subset M$ such that $\int_\Sigma\zeta=t$ for some $0<t<t_0$. Let $\gamma$ be a constant loop at some point $p\in\Sigma$, and $D=\Sigma\setminus\{p\}$. Then $\partial D=\gamma$ so that $hol(\gamma)=-\mathrm{exp}\Big( \int_D \zeta \cdot \xi \Big)$. But $hol(\gamma)=e$ as $\gamma$ is the identity loop. On the other hand, $\int_D \zeta=\int_\Sigma\zeta=t$. This implies $\mathrm{exp}(t\xi)=e$, which is a contradiction.

Therefore, $\xi$ must be zero and the connection is flat.
\end{proof}

The classification of $G$-bundles with flat connections can be represented by the conjugacy classes of morphisms from $\pi_1(M)\to G$ (c.f. \cite{Morita} Theorem 2.9). So we have proved the theorem in this case.

Below we will demonstrate the theorem for $M=T^4$ and $G=U(1)=S^1$.

\begin{ex}
We consider principal $U(1)$ bundles or circle bundles over $T^4$. We describe $T^4$ as $\mathbb{R}^4/\sim$, with the identification $(x_1,x_2,x_3,x_4)\sim(x_1+a,x_2+b,x_3+c,x_4+d),a,b,c,d\in\mathbb{Z}$, and $S^1$ as $\{ z\in\mathbb{C}\mid |z|=1 \}$. Let $\zeta=c_1dx_1\wedge dx_2+c_2 dx_3\wedge dx_4$ be a closed 2-form with $c_1,c_2\in\mathbb{R}\setminus \{0\}$. Then the equivalent classes of circle bundles with a $\zeta$-flat connection are the conjugacy classes of morphisms $\rho:\Gamma\to S^1$, where $\Gamma=\Psi^a_a(M)/\Psi^a_{a,\zeta}(M)$. Since $S^1$ is an abelian group, this classification of $\zeta$-flat connections is just the morphisms $\rho:\Gamma\to S^1$.

Since $\pi_2(T^4)$ is trivial, $\Gamma$ is an $\mathbb{R}$-extension of $\pi_1(T^4)$. To describe its group structure explicitly, let $a_i$ be the straight line path in $\mathbb{R}^4$ starting at the origin and ending at $(0,\ldots,1,\ldots,0)$, where the $i$-th coordinate is 1 and the other coordinates are 0. When projected to $T^4$, $a_1,\ldots,a_4$ become the generators of $\pi_1(T^4)$. Although $\pi_1(T^4)$ is abelian, $a_1,\ldots,a_4$ may no longer commute in its $\mathbb{R}$-extension $\Gamma$. That is, $a_j^{-1}a_i^{-1}a_ja_i$ is some non-trivial class of a contractible loop. For each contractible loop $b$ of $T^4$, there is an oriented disk $D$ such that $\partial D=b$, and $\partial D$ has the same orientation as $b$. Let $|b|=\int_D \zeta$, then $|b|$ is independent of the choice of $D$. Moreover, contractible loops $b$ and $b'$ represent the same class in $\Gamma$ if and only if $|b|=|b'|$. So a class of contractible loops $b\in\Gamma$ can be represented by a real number $|b|$. Therefore, $\Gamma$ is generated by $a_1,\ldots,a_4,b$ and their multiplication are defined as
\begin{align*}
    & |a_2^{-1}a_1^{-1}a_2a_1|=c_1, \\
    & |a_4^{-1}a_3^{-1}a_4a_3|=c_2, \\
    & |a_j^{-1}a_i^{-1}a_ja_i|=0, \text{ for other }i,j, \\
    & a_ib=ba_i \\
    & |b'b|=|b|+|b'|.
\end{align*}

The possible morphisms of $\rho:\Gamma\to S^1$ is dependent on whether $\frac{c_1}{c_2}\in\mathbb{Q}$. Since $S^1$ is abelian, 
$$
\rho(c_1)=\rho(a_2^{-1})\rho(a_1^{-1})\rho(a_2)\rho(a_1)=\rho(a_2)^{-1}\rho(a_2)\rho(a_1)^{-1}\rho(a_1)=1.
$$
Similarly $\rho(c_2)=1$, so $\rho(pc_1+qc_2)=1$ for any $p,q\in\mathbb{Z}$.

When $\frac{c_1}{c_2}\in\mathbb{Q}$, $\{ \rho(pc_1+qc_2)| p,q\in\mathbb{Z} \}$ has a minimal positive number $c_0$. Then $\rho(b)$ must have the form $e^{\frac{2\pi n|b|i}{c_0}}$ for some $n\in\mathbb{Z}$. If $\rho(b)=e^{\frac{2\pi n|b|i}{c_0}}$, the Euler class of the circle bundle is $\frac{n}{c_0}\zeta$, and the definition of $\rho(a_i)$ determines the connection chosen.

When $\frac{c_1}{c_2}\notin\mathbb{Q}$, $\{ \rho(pc_1+qc_2)| p,q\in\mathbb{Z} \}$ is dense in $\mathbb{R}$. So $\rho(b)$ must be 1 for any $|b|\in\mathbb{R}$. Thus, the classification is only dependent on $\rho(a_i)$, and becomes equivalent to the classification of flat connections. Actually, the Euler class in this case is $c\,\zeta$ for some $c\in\mathbb{R}$ but it must be integral. So the only possible $c$ is 0, i.e. every $\zeta$-flat connection is flat.
\end{ex}
\begin{rmk}The above example shows that different $\zeta$ can result in different classifications of $\zeta$-flat principal bundles. In particular, for symplectic manifolds, the classification of symplectically flat bundles can vary with the symplectic structure.
\end{rmk}

\section{Curvature functionals whose zeroes are symplectically flat}

Let $P$ be a principal $G$-bundle over $M$.  An inner product on the Lie algebra of the structure group $G$ together with a Riemannian metric $g$ on $M$ induces an inner product $\langle-,-\rangle$ on $\Omega^*(M,Ad\,P)$.  The squared norm of the curvature $F\in \Omega^2(M, Ad\,P)$ gives the Yang-Mills functional 
\begin{align}\label{YMf}
\|F\|^2=\langle F,F \rangle=\int_M F\wedge *F\,.
\end{align}
Flat connections, i.e. $F(A)=0$, are the zeroes of the Yang-Mills functional.  In this section, we will write down functionals whose zeroes correspond to symplectically-flat and $\zeta$-flat connections and study some of their properties.  We will describe first functionals related to symplectically flat connections.  But before doing so, we will set our conventions and give a very brief review of the differentials of the TTY algebra \cite{TTY} of primitive forms on symplectic manifolds and their twisting.  Further details can be found in \cite{TZ}*{Section 2}.

\subsection{Review of the twisted differentials of the TTY-algebra}
Let $(M^{2n},\omega)$ be a symplectic manifold. Under the Lefschetz decomposition, a differential form $\eta_k\in \Om^k(M)$ can be expressed as a polynomial in $\om$:
\begin{align}\label{Lefd}
\eta_k= \beta_k + \om \w \beta_{k-2} + \ldots + \om^p \w \beta_{k-2p}+ \ldots
\end{align}
where $\{\beta_k, \beta_{k-2}, \ldots, \beta_{k-2p}, \ldots\}$ are primitive forms determined uniquely by $\eta_k$ and $\om$. We use $P^k(M)$ to denote the space of primitive $k$-forms, and $\Pi:\Omega^k(M)\to P^k(M)$ to denote the projection of forms into their primitive component.

Let
\begin{align*}
L&:\Omega^k(M)\to~\Omega^{k+2}(M)\\
&\qquad \alpha\quad~ \mapsto\quad \omega\wedge\alpha
\end{align*}
be the Lefschetz operator. Given $g$ a Riemannian metric which is compatible with $\omega$, we have the dual Lefschetz operator $\Lambda=L^*$. Actually, $\Lambda$ is independent of the choice of the metric and can be defined without it. The kernel of $\Lambda$ is exactly $P^*(M)$.

In \cite{TY2}, Tseng and Yau observed that $d[\om^rP^k(M)]\subset \om^rP^{k+1}(M)\oplus \omega^{r+1} P^{k-1}(M)$. This implies the decomposition of the exterior derivative $d = \partial_++\omega\w \partial_-$ where
$$
\partial_+:\om^rP^k(M)\to \om^rP^{k+1}(M), \qquad \partial_-:\om^rP^k(M)\to \om^rP^{k-1}(M)\,
$$
and ($\dpp$, $\dpm$) satisfy
$$
\dpp^2=0\,,\quad\dpm^2=0\,,\quad  \om\w (\dpp\dpm+\dpm\dpp)=0\,.
$$
Based on this, they constructed an elliptic complex of primitive forms.
\begin{align}\label{primcomplex}
\xymatrix@R=30pt@C=30pt{
0\; \ar[r] & \; P^{0}(M) \ar[r]^\dpp &\; P^{1}(M) \ar[r]^\dpp& ~ \ldots ~\ar[r]^\dpp&\; P^{n-1}(M) \ar[r]^\dpp&\; P^{n}(M) \ar[d]^{-\dpp\dpm}\\
0\; & \; P^{0}(M) \ar[l]_{-\dpm} &\; P^{1}(M)\ar[l]_{~~-\dpm}& ~ \ldots ~\ar[l]_{~~~-\dpm}&\; P^{n-1}(M) \ar[l]_{-\dpm}&\; P^{n}(M) \ar[l]_{~~~-\dpm} \; 
}
\end{align}

In \cite{TZ}*{Section 2.2}, we showed that the covariant derivative $d_A$ when acting on the space of twisted forms $\Omega^*(M,Ad\,P)$ also has a similar decomposition $d_A=\dpa+\om\w\dma$ where
$$
\dpa:\om^rP^k(M,Ad\,P)\to \om^rP^{k+1}(M,Ad\,P), \qquad \dma:\om^rP^k(M,Ad\,P)\to \om^rP^{k-1}(M,Ad\,P).
$$
But in general, $\partial_{+A}^2\,,\,\partial_{-A}^2$ and $\om\w(\dpa\dma+\dma\dpa)$ are non-trivial and dependent on the curvature $F=dA + A\w A$.

\subsection{Primitive Yang-Mills functional}


On a symplectic manifold $(M^{2n},\omega)$, the curvature 2-form $F$ can be Lefschetz decomposed and expressed as $F=F_p+\Phi\,\omega$, where $F_p\in P^2(M, Ad\,P)$ and $\Phi\in \Omega^0(M, Ad\,P)$.  In this section, we will assume that $\dim M=2n\geq 4$ so that $F_p$ is non-trivial.  With the Lefschetz decomposition of $F$, the Yang-Mills functional \eqref{YMf} decomposes into the sum of two components
$$
\|F\|^2=\|F_p\|^2+\|\Phi\,\omega\|^2.
$$

As we shall see, a critical point of the Yang-Mills functional may not necessarily be a critical point of both $\|F_p\|^2$ and $\|\Phi\,\omega\|^2$.  Hence, it is interesting to consider these two functionals separately.  The first component $\|F_p\|^2$ is the primitive Yang-Mills functional introduced in Definition \ref{dpYM}.  It turns out to have many properties similar to those of the Yang-Mills functional.  Due to this similarity, our description here will mirror that of Atiyah-Bott's for Yang-Mills connections in \cite{Atiyah-Bott}*{Section 4}.

To begin, recall that a connection $A$ is a Yang-Mills connection if and only if $d^*_A F=0$. The equation for the primitive Yang-Mills connections is similar.
\begin{prop}
Let $P$ be a principal bundle over a symplectic manifold $(M,\omega)$, and $A$ is a connection over $P$. Then $A$ is a primitive Yang-Mills connection if and only if $d_A^* F_p=-*d_A*F_p=0$.
\end{prop}

\begin{proof}
Let $\eta\in\Omega^1(M,Ad\,P)$ be an infinitesimal variation of $A$. Set $A_t=A+t\eta$ and $F_t$ as the curvature of $A_t$. Then
$$
F_t=F(A)+td_A\eta+\frac{1}{2}t^2[\eta,\eta],
$$
and its primitive part
$$
F_p(A_t)=F_p(A)+t\partial_{+A}\eta+\frac{1}{2}t^2\Pi[\eta,\eta].
$$
So we have
\begin{align}\label{infinitesimal of F_p^2}
    \|F_p(A_t)\|^2=\|F_p(A)\|^2+2t\langle\partial_{+A}\eta,F_p(A)\rangle+t^2\left(\|\partial_{+A}\eta\|^2+\langle F_p(A),\Pi[\eta,\eta] \rangle\right)+o(t^2)
\end{align}

Thus, $A$ is a critical point of $\|F_p\|^2$ is equivalent to $\langle\partial_{+A}\eta,F_p(A)\rangle=0$ for any $\eta$. Moreover, since $\omega\w\partial_{-A}\eta$ is orthogonal to primitive forms, we have
\begin{align}\label{dA* on primitive 2-form}
    \langle\eta,d_A^*F_p(A)\rangle=\langle d_A\eta,F_p(A)\rangle=\langle\partial_{+A}\eta,F_p(A)\rangle\,.
\end{align}
This implies that $A$ is a primitive Yang-Mills connection if and only if  $d_A^* F_p(A)=0$.

Finally, on even dimensional manifolds, $d_A^*=-*d_A*$. This completes the proof.
\end{proof}

\begin{rmk}\label{harmonic}
Observe that $\partial_{+A}^*F_p=d_A^*F_p$ according to \eqref{dA* on primitive 2-form}. On the other hand,
$$
\partial_{+A}F_p=\Pi\circ d_AF=0.
$$
Thus, $A$ is a primitive Yang-Mills connection if and only if $F_p$ is a harmonic form of the operator $\partial_{+A}$. This is analogous to the statement that $A$ is a Yang-Mills connection if and only if $F$ is a harmonic form of $d_A$.
\end{rmk}

For the Yang-Mills functional, the Hessian of the functional at a critical point gives the following quadratic form:
$$
Q(\eta,\eta)=\langle d_A^*d_A\eta+*[*F,\eta],\eta \rangle.
$$
The quadratic form for the primitive Yang-Mills functional has a similar form.
\begin{prop}
The quadratic form defined by the Hessian of the primitive Yang-Mills functional at a critical point is given by
$$
Q_p(\eta,\eta)=\langle d_A^*\partial_{+A}\eta+*[*F_p,\eta],\eta \rangle.
$$
\end{prop}

\begin{proof}
By \eqref{infinitesimal of F_p^2},
$$
Q_p(\eta,\eta)=\|\partial_{+A}\eta\|^2+\langle F_p(A),\Pi[\eta,\eta] \rangle.
$$
As $\partial_{+A}\eta$ is a primitive 2-form, it is orthogonal to $\omega\w \partial_{-A}\eta$. So
$$
\|\partial_{+A}\eta\|^2=\langle \partial_{+A}\eta,d_A\eta \rangle=\langle d_A^*\partial_{+A}\eta,\eta \rangle.
$$
For the remaining term, we have
\begin{align*}
    \langle F_p(A),\Pi[\eta,\eta] \rangle &= \langle F_p(A),[\eta,\eta] \rangle \\
    &= \int_M [\eta,\eta]\wedge *F_p \\
    &= \int_M \eta\wedge[\eta,*F_p] \\
    &= -\int_M \eta\wedge **^{-1}[*F_p,\eta] \\
    &= \int_M \eta\wedge **[*F_p,\eta] \\
    &= \langle *[*F_p,\eta],\eta \rangle.
\end{align*}
The first line of equation holds because the primitive 2-form $F_p(A)$ is orthogonal to $\omega$. The fifth line holds because $[*F_p,\eta]$ is a 1-form and $*=-*^{-1}$ when acting on odd-degree forms.
\end{proof}

$Q(\eta,\eta)$ induces an operator $L_A=d_A^*d_A+*[*F,-]$ on $\Omega^*(M,Ad\,P)$. $\eta\in\Omega^1(M, Ad\,P)$ is a tangent vector on the space of Yang-Mills connections at $A$ if and only if $L_A\eta=0$. This also works for the operator induced by $Q_p(\eta,\eta)$.

\begin{thm}
Let $(L_p)_A=d_A^*\partial_{+A}+*[*F_p,-]$ be the operator $\Omega^*(M,Ad\,P)$ induced by $Q_p(\eta,\eta)$. Suppose $A$ is a primitive Yang-Mills connection, then $\eta\in\Omega^1(M, Ad\,P)$ is a tangent vector on the space of primitive Yang-Mills connections if and only if $(L_p)_A\eta=0$.
\end{thm}

\begin{proof}
Let $A_t=A+t\eta+o(t)$. Since $\frac{\partial}{\partial t}(d_{A_t})|_{t=0}=[\eta,-]$ and $\frac{\partial}{\partial t}\big(F_p(A_t)\big)|_{t=0}=\partial_{+A}\eta$, we have
$$
\frac{\partial}{\partial t}\big(*d_{A_t}*F_p(A_t)\big)|_{t=0}=*[\eta,*F_p(A)]+*d_A*(\partial_{+A}\eta)=-(L_p)_A\eta.
$$
So $A_t$ is a path on the space of primitive Yang-Mills connections if and only if $(L_p)_A\eta=0$.
\end{proof}

As we mentioned in Remark \ref{harmonic}, $\partial_{+A}^*=d_A^*$ when acting on primitive 2-forms. So $(L_p)_A$ can also be written as $\partial_{+A}^*\partial_{+A}+*[*F_p,-]$. On the other hand, $A_t=A+t\eta+o(t)$ are all gauge equivalent to $A$ if and only if $\eta$ is $d_A$-exact. As $d_A=\dpa$ when acting on 0-forms, the complement of $d_A$-exact forms can be chosen as $\partial_{A+}^*$-closed forms. Thus, the tangent space of primitive Yang-Mills connections up to gauge equivalence can be identified with the space of $\eta\in\Omega^1(M,Ad\,P)$ satisfying $(L_p)_A\eta=0$ and $\partial_{A+}^*\eta=0$. Or equivalently, $\partial_{A+}^*\eta=0$ and
$$
\partial_{+A}^*\partial_{+A}\eta+\partial_{+A}\partial_{A+}^*\eta+*[*F_p,\eta]=0.
$$
The above operator acting on $\eta$ is elliptic because $\eqref{primcomplex}$ is an elliptic complex. Therefore, the solution space of primitive Yang-Mills connections up to gauge equivalence is finite dimensional.

By the same arguments, we can prove similar properties for $\|\Phi\,\omega\|^2$:
\begin{itemize}
    \item  $A$ is a critical point of $\|\Phi\,\omega\|^2$ if and only if $d_A^*(\Phi\,\omega)=0$. Moreover,
    $$
    d_A^*(\Phi\,\omega)=-*d_A*(\Phi\,\omega)=-\frac{1}{(n-1)!}*d_A(\Phi\,\omega^{n-1})=-\frac{1}{(n-1)!}*[(d_A\Phi)\omega^{n-1}].
    $$
    So these equivalent statements hold if and only if $(d_A\Phi)\omega^{n-1}=0$. But since the map $\omega^{n-1}\!:\Omega^1\to\Omega^{2n-1}$ is an isomorphism, an equivalent condition for a connection $A$ to be a critical point of $\|\Phi\,\omega\|^2$ is $d_A\Phi=0\,$.
    \item The quadratic form defined by the Hessian of $\|\Phi\,\omega\|^2$ at a critical point is given by
    $$
    Q_{\Phi}(\eta,\eta)=\langle d_A^*(\omega\partial_{-A}\eta)+*[*(\Phi\,\omega),\eta],\eta \rangle.
    $$
    \item Let $(L_{\Phi})_A=d_A^*(\omega\partial_{-A})+*[*(\Phi\,\omega),-]$ be the operator acting on $\Omega^*(M,Ad\,P)$ induced by $Q_{\Phi}(\eta,\eta)$. Suppose $A$ is a critical point of $\|\Phi\,\omega\|^2$, then $\eta$ is a tangent vector on the space of the critical points of $\|\Phi\,\omega\|^2$ if and only if $(L_{\Phi})_A\eta=0$.
\end{itemize}

\begin{rmk}
However, the critical points of $\|\Phi\,\omega\|^2$ modulo gauge equivalence can not be expressed as the kernel of some elliptic operator, as $d_A\Phi=0$ is a weak condition. This makes the functional $\|F_p\|^2$ more appealing as compared to $\|\Phi\,\omega\|^2$.
\end{rmk}

To summarize, the conditions for the critical points of the three functionals -- $\|F\|^2$, $\|F_p\|^2$, and $\|\Phi\,\omega\|^2$ -- are given in Table \ref{Tab1}. These conditions immediately give the following proposition.

\begin{table}[t!]
\centering
\begin{tabular}{| c | c |}
\hline
    Functional & Critical point $A$ \\\hline
    $\|F\|^2$ \  (Yang-Mills)& $d_A^*F=0$ \\\hline
    $\|F_p\|^2$\ (primitive Yang-Mills) & $d_A^*F_p=0$ \\\hline
    $\|\Phi\,\omega\|^2$ & $d_A^*(\Phi\,\omega)=0$, or equivalently $d_A\Phi=0$. \\ \hline
\end{tabular}
\caption{Three functionals and the conditions for their critical points.}\label{Tab1}
\end{table}

\begin{prop}\label{2 of 3}
If $A$ is a critical point of two of the functionals among $\|F\|^2,\|F_p\|^2$ and $\|\Phi\,\omega\|^2$, it is also a critical point of the third one.
\end{prop}

In the following, we will explore the relations between the critical points of $\|F_p\|^2$ versus that of  $\|F\|^2$. 

\begin{ex}
When $A$ is a symplectically flat connection, it is Yang-Mills. Since $F_p=0$ in this case, it is a primitive Yang-Mills connection, hence also a critical point of $\|\Phi\,\omega\|^2$.
\end{ex}

\begin{ex}
A Hermitian Yang-Mills connection $A$ is a Yang-Mills connection. Its curvature can be written as $F_p^{1,1}+c\,I\omega$, where $F_p^{1,1}$ is a primitive $(1,1)$-form, $c$ is a constant and $I$ is the identity map on $Ad\,P$. So $\Phi=c\,I$ is $d_A$-closed, which implies $A$ is a critical point of $\|\Phi\,\omega\|^2$.  Therefore, it is also a primitive Yang-Mills connection. 
\end{ex}

More generally, we have the following statement.

\begin{lem}\label{YM is PYM}
Suppose $P$ is a principal $G$-bundle over a closed manifold $M$ with a connection $A$. Let $\Lambda$ be the dual Lefschetz operator and $\Delta_A=d_Ad_A^*+d_A^*d_A$. When the curvature $F$ satisfies $[\Lambda,\Delta_A]$F=0, $A$ is Yang-Mills implies that it is a primitive Yang-Mills connection. If in addition $G$ is an abelian group,  $A$ is a primitive Yang-Mills connection also implies that it is Yang-Mills.
\end{lem}

\begin{proof}
When $A$ is Yang-Mills, $\Delta_A F=0$. Write $F=F_p+\Phi\,\omega$ by the Lefschetz decomposition. Then we have
$$
d_A^*d_A\Phi=\Delta_A\Phi=\Delta_A\Lambda F=\Lambda\Delta_A F=0.
$$
It follows that $d_A\Phi=0$. By Proposition \ref{2 of 3}, $A$ is also a primitive Yang-Mills connection.

Conversely, when $A$ is a primitive Yang-Mills connection, $F_p$ is a harmonic form of $\partial_{+A}$. If in addition $G$ is abelian, $F$ and $F_p$ are invariant under each fiber. So they can be treated simply as 2-forms over $M$. In this case $d_A$ and $\partial_{+A}$ are just $d$ and $\partial_+$, respectively. Then we can say that $F_p\in P^2(M)$ is a harmonic form of $\partial_+$.

On the other hand, since $F\in\Omega^2(M)$ is $d$-closed, there exists some $\xi\in\Omega^1(M)$ such that $F+d\xi$ is a harmonic form of $d$. By the discussion of the previous part, we have $F_p+\partial_+\xi=\Pi(F+d\xi)$ is also $\partial_{+A}$-harmonic.

When $\dim M\geq6$, by the ellipticity of \eqref{primcomplex}, the $\partial_{+A}$-harmonic forms, $F_p$ and $F_p+\partial_+\xi$, must be identical. This implies that $d\xi=f\om$ for some function $f\in\Om^0(M)$. Then $df\wedge\om=d^2\xi=0$ and $f$ must be a constant. Since $[\om]\in H^2(M)$ is non-trivial for closed manifolds, $f=0$. Thus, $F$ is harmonic and $A$ is Yang-Mills.

When $\dim M=4$, to use the ellipticity of \ref{primcomplex}, we need to show that $F_p$ and $F_p+\partial_+\xi$ are $\partial_+\partial_-$ and $\partial_+^*$-closed, where the latter closeness has been proved. As $dF=0$, $\om(\dpm F_p+\dpp \Phi)=0$. So $\dpp\dpm F_p=-\dpp^2\Phi=0$. Similarly $\dpp\dpm (F_p+\partial_+\xi)=0$. Then applying the same arguments as above, we have $F_p=F_p+\partial_+\xi$ and it leads to that $A$ is Yang-Mills.
\end{proof}

\begin{cor}
For a principal bundle over a K\"ahler manifold, we have the Hodge decomposition $d_A=\partial_A+\bar{\partial}_A$. If the curvature of a Yang-Mills connection satisfies $\left[(\partial_A^*)^2-(\bar{\partial}_A^*)^2\right]F=0$, then this connection is a primitive Yang-Mills connection.
\end{cor}
\begin{proof}
$$
[\Lambda,\Delta_A]=[\Lambda,[d_A,d_A^*]]=[[\Lambda,d_A],d_A^*]+[d_A,[\Lambda,d_A^*]].
$$
The second term is 0 because for any $\xi,\eta\in\Omega^*(M,Ad\,P)$, we have
$$
\langle[\Lambda,d_A^*]\xi,\eta\rangle=\langle\xi,[d_A,L]\eta\rangle=0.
$$
For the first term, by the K\"ahler identities, we have
$$
[[\Lambda,d_A],d_A^*]=[i(\bar{\partial}_A^*-\partial_A^*),\partial_A^*+\bar{\partial}_A^*]=2i[(\partial_A^*)^2-(\bar{\partial}_A^*)^2].
$$
So by assumption it vanishes when acting on $F$. In this case, a Yang-Mills connection is a primitive Yang-Mills connection.
\end{proof}

If in addition the structure group is abelian, $\partial_A^*$ and $\bar{\partial}_A^*$ acting on $F\in\Om^2(M,Ad\,P)$ can be viewed as $\partial^*$ and $\bar{\partial}^*$ acting on $F\in\Om^2(M)$ respectively. As $(\partial^*)^2=(\bar{\partial}^*)^2=0$, we have the following statement.

\begin{cor}
If $P$ is a principal bundle over a K\"ahler manifold with an abelian structure group, then the Yang-Mills connections and primitive Yang-Mills connections are the same.
\end{cor}

In the following, we give two examples that make clear that the critical points of the Yang-Mills and primitive Yang-Mills functionals are in general different.  The first demonstrate this for abelian solutions on $T^4$.  The second shows that the BPST  Yang-Mills instanton solutions on $\mathbb{R}^4$ \cite{BPST} is not a primitive Yang-Mills solution with respect to the standard symplectic structure.

\begin{ex}\label{T^4}
Let $T^4=\mathbb{R}^4/2\pi\mathbb{Z}^4$ be a 4-torus, and $\om=dx_1\wedge dx_2+dx_3\wedge dx_4$ be its symplectic form. Take the Riemannian metric to be 
$$
g=dx_1^2+dx_2^2+\frac{1}{f}dx_3^2+fdx_4^2
$$
where $f=\dfrac{3+2\sin 2x_2\cos x_3}{1-\frac{1}{2}\sin 2x_2\cos x_3}$. Construct a circle bundle $X=\mathbb{R}^5/\sim$ over $T^4$ by identifying
\begin{align*}
x_i\sim x_i+2\pi n_i \,,\qquad 
y\sim y+2\pi n_5-n_2x_3\,
\end{align*}
for $i=1, \ldots, 4,$ and $n_i, n_5\in\mathbb{Z}$, Set the connection as
$$
A=dy+\frac{1}{2\pi}x_2dx_3+\frac{1}{4\pi}\sin 2x_2\sin x_3 dx_1.
$$
Then the curvature
$$
F=-\frac{1}{2\pi}\cos 2x_2\sin x_3 dx_1dx_2+\frac{1}{2\pi}(1-\frac{1}{2}\sin 2x_2\cos x_3)dx_1dx_3
$$
can be treated as a 2-form over $M$. As
\begin{align*}
    d*F &= \frac{1}{2\pi}d[-\cos 2x_2\sin x_3 dx_3dx_4-(1-\frac{1}{2}\sin 2x_2\cos x_3)fdx_2dx_4] \\
    &= \frac{1}{2\pi}d[-\cos 2x_2\sin x_3 dx_3dx_4-(3+2\sin 2x_2\cos x_3)dx_2dx_4] \\
    &= 0,
\end{align*}
$F$ is $d$-harmonic so that $A$ is Yang-Mills. However, under the Lefschetz decomposition $F=F_p+\Phi\,\om$, we have $\Phi=-\frac{1}{4\pi}\cos 2x_2\sin x_3$, which is not $d$-closed. This implies that $A$ is not a primitive Yang-Mills connection.

On the other hand, by the discussion in the last paragraph of the proof of Lemma \ref{YM is PYM}, $F_p\in P^2(M)$ is $\dpp\dpm$-closed. By the ellipticity of \eqref{primcomplex}, there exists some $\xi\in P^1(M)$ such that $F_p+\dpp\xi$ is $\partial_+^*$-closed. Consider the connection $A'=A+\xi$. Its curvature is $F'=F+d\xi$ and the primitive part is exactly $F'_p=F_p+\dpp\xi$. Hence, $A'$ is a primitive Yang-Mills connection. But by $\partial_+^*F_p\neq0$, $\dpp\xi$ can not be zero. Then $d\xi\neq0$ and $F'=F+d\xi$ is not $d$-harmonic. Thus, $A'$ is not Yang-Mills.
\end{ex}

\begin{ex}
Let $P$ be an $SU(2)$-bundle over $M=\mathbb{R}^4$ and $\omega=dx^1\wedge dx^2+dx^3\wedge dx^4$. Set
$$
A=\frac{\eta^a_{\mu\nu} x^{\nu}}{|x|^2+1}T_a dx^{\mu}
$$
be a BPST instanton \cite{BPST} (see also the review article \cite{BVV}). Here $\{ T_1,T_2,T_3 \}$ is a basis of $\mathfrak{su}(2)$ satisfying $[T_a,T_b]=\epsilon^{abc}T_c$. For example, we can set $T_a=\frac{1}{2i}\sigma_a$ for the Pauli matrices $\sigma_1,\sigma_2,\sigma_3$. $\eta^a_{\mu\nu}$ is the
't Hooft symbol.
$$
\eta^a_{\mu\nu}=
\begin{cases}
\epsilon^{a\mu\nu}, & \mu,\nu=1,2,3 \\
-\delta^{a\nu}, & \mu=4 \\
\delta^{a\mu}, & \nu=4 \\
0, & \mu=\nu=4
\end{cases}.
$$
The curvature is
$$
F=\frac{4}{(|x|^2+1)^2}[-T_1(dx^1\wedge dx^4+dx^2\wedge dx^3)+T_2(dx^1\wedge dx^3+dx^2\wedge dx^4)-T_3(dx^1\wedge dx^2+dx^3\wedge dx^4)].
$$
$F$ is self-adjoint so $A$ is a Yang-Mills connection. On the other hand, we have
$$
\Phi=-\frac{4}{(|x|^2+1)^2}T_3.
$$
Since $\frac{4}{(|x|^2+1)^2}$ is not $d$-closed, $d\Phi$ has a non-trivial $T_3$-component. But $[T_a,T_3]$ has no $T_3$-component for any $a$, neither is $[A,\Phi]$. Therefore, $d_A\Phi$ has a non-trivial $T_3$-component and cannot be 0. $A$ is not a critical point of $\|\Phi\,\omega\|^2$, so not a primitive Yang-Mills connection, either.
\end{ex}
\begin{rmk}
The construction of BPST instanton makes use of the homotopy of $S^3\rightarrow S^3$, mapping from the boundary of $\mathbb{R}^4 \rightarrow SU(2)=S^3$ \cite{BVV}. However, 
the presence of $\om$ in the  primitive Yang-Mills conditions breaks the symmetry of $\mathbb{R}^4$ coordinates by separating them into two groupings of $\{x^1, x^2\}$ and $\{x^3, x^4\}$. 
\end{rmk}

\subsection{Cone Yang-Mills functional}

Suppose $(M^{2n},\omega)$ is a symplectic manifold and $[\omega]\in H^2(M, \mathbb{Z})$. We can construct a circle bundle $\pi:X\to M$ whose Euler class is $\omega$. Let $P$ be a principal bundle over $M$ with a symplectically flat connection $A$. Then $A-\theta\,\Phi$ is a flat connection on the pullback bundle $\pi^*P$ \cite{TZ}*{Corollary 3.9}. Given an inner product structure over the associated vector bundle $Ad\,P$ compatible with $\omega$, we can extend it to an inner product over $\pi^*Ad\,P$ and make the global angular form $\theta\in\Omega^1(X)$ orthogonal to $\pi^*\Omega^*(M)$. These two inner products induce Yang-Mills functionals on $Ad\,P$ and $\pi^*Ad\,P$ respectively, and they have the following relationship.

$$
S_{Y\!M}(A-\theta\Phi)=\|F(A)-\omega\Phi-\theta d_A\Phi\|^2_{\pi^*Ad\,P}=\|F(A)-\omega\Phi\|^2_{Ad\,P}+\|\theta d_A\Phi\|^2_{Ad\,P}.
$$

Observe that the right hand side is the cone Yang-Mills functional defined in Definition \ref{def of cYM} if we replace $-\Phi$ by a function $B\in\Omega^0(M,Ad\,P)$, and is well defined even if $[\omega]$ is not an integral class. This suggests the consideration of the mapping cone of differential forms $\Om^*(M)[\theta]=\Om^*(M)\oplus\,\theta\,\Om^*(M)$, which is quasi-isomorphic to both $\Om^*(X)$ and the space of primitive forms  \cite{TT}. Here, $\theta$ has degree 1 and satisfies $d\theta=\om$, and further, $[\omega]$ is allowed to be non-integral. For a connection $A$ on a principal bundle $P$ over $M$, its covariant derivative $d_A$ induces an operator $d_A-\theta\Phi$ over $\Om^*(M,Ad\,P)[\theta]$. Then $(d_A-\theta\Phi)^2=0$ if and only if $A$ is a symplectically flat connection \cite{TZ}*{Proposition 3.8}.

We can also extend the inner product $\langle-,-\rangle$ on $\Om^*(M,Ad\,P)$ to $\langle-,-\rangle_{\mathcal{C}}$ on $\Om^*(M,Ad\,P)[\theta]$, by setting
$$
\langle \xi_1+\theta\eta_1, \xi_2+\theta\eta_2\rangle_{\mathcal{C}}=\langle\xi_1,\xi_2\rangle+\langle\eta_1,\eta_2\rangle.
$$
It induces an operator
\begin{align*}
*_{\mathcal{C}}:\Om^*(M,Ad\,P)[\theta]~&\to~~\Om^*(M,Ad\,P)[\theta]\\ \xi+\theta\eta~\qquad&\mapsto~~ (-1)^{|\xi|}\theta *\xi+*\eta
\end{align*}
satisfying
$$
\langle \xi_1+\theta\eta_1, \xi_2+\theta\eta_2\rangle_{\mathcal{C}}=\int_M\frac{\partial}{\partial\theta}[(\xi_1+\theta\eta_1)\wedge *_{\mathcal{C}}(\xi_2+\theta\eta_2)]
$$
where $\frac{\partial}{\partial\theta}(\theta\alpha)=\alpha$ for any $\alpha\in\Om^*(M,Ad\,P)$.

Given a connection $A$ and a function $B\in\Om^0(M,Ad\,P)$, the curvature of the operator $d_A+\theta B$ is $\tilde{F}=(d_A+\theta B)^2=F(A)+\om B-\theta d_A B$. It suggests the following cone Yang-Mills functional for this curvature
$$\|\tilde{F}\|_{\mathcal{C}}^2=\|F(A)+\om B\|^2+\|d_AB\|^2.$$
The zero points of the functional  satisfies
$$
F=-\om B,\quad d_AB=0,
$$
which identifies $A$ as a symplectically flat connection. 

A straightforward calculation shows that $A+\theta B$ is a critical point of $\|\tilde{F}\|_{\mathcal{C}}^2$ if and only if
$$
\begin{cases}
d_A^*(F_A+\om B)-[d_AB,B] = 0, \\
n(\Lambda F_A+B)+d_A^*d_AB =0.
\end{cases}
$$

We may ask for what connection $A$ there is some $B$ such that $A+\theta B$ is a critical point of the cone Yang-Mills functional. For a connection which is both Yang-Mills and primitive Yang-Mills, $B=-\Phi=-\Lambda F_A$ is an obvious solution. However, if only one of the condition above holds, there may not be solutions.

\begin{prop}
Suppose the structure group is abelian. If a connection $A$ is Yang-Mills but not  primitive Yang-Mills, or is primitive Yang-Mills but not Yang-Mills, then there is no $B$ such that $A+\theta B$ is a critical point of the cone Yang-Mills functional.
\end{prop}
\begin{proof}
As the structure group is abelian, the condition that $A+\theta B$ is a critical point of the cone Yang-Mills functional $\|\tilde{F}\|_{\mathcal{C}}^2$ becomes
\begin{align}\label{critical pt of abelian cone YM}
\begin{cases}
d^*(F_A+\om B)= 0, \\
n(\Lambda F_A+B)+d^*dB =0.
\end{cases}
\end{align}

Write $F_A=F_p+\Phi\om$ by the Lefschetz decomposition.
In both cases $d\Phi\neq 0$ because of Proposition \ref{2 of 3}.

When $A$ is Yang-Mills, the first equation of \eqref{critical pt of abelian cone YM} is equivalent to $d*(\om B)=\frac{\om^{n-1}}{(n-1)!} dB=0$. So $dB=0$ and $B$ is a constant. Then the second equation implies that $\Phi=\Lambda F_A$ is a constant, which is contradicted to $d\Phi\neq 0$. Thus, such $B$ does not exist.

When $A$ is primitive Yang-Mills, $d^*$ acting on $F_p$ is zero, the first equation of \eqref{critical pt of abelian cone YM} is equivalent to $d*(\om\, \Phi+\om\ B)=\frac{\om^{n-1}}{(n-1)!} d(\Phi+B)=0$. Thus, $\Phi+B=\Lambda F_A+B$ must be a constant then so is $d^*dB$. Since the only $d^*$-exact constant is zero, $d^*dB$ must vanish. It follows that $B=-\Phi$ and $dB=0$. This contradicts $d\Phi\neq 0$, so no such $B$ exists either.
\end{proof}

On the other hand, it is easy to show that $B$ is unique for abelian structure group.

\begin{prop}
Suppose the structure group is abelian. For each connection $A$ there is at most one $B$ such that $A+\theta B$ is a critical point of the cone Yang-Mills functional.
\end{prop}
\begin{proof}
Suppose both $A+\theta B_1$ and $A+\theta B_2$ are critical points of $\|\tilde{F}\|_{\mathcal{C}}^2$. Then $(A,B_1)$ and $(A,B_2)$ both need to satisfy \eqref{critical pt of abelian cone YM}. By the first equation of \eqref{critical pt of abelian cone YM} we have
$$
\frac{\om^{n-1}}{(n-1)!}dB_1=d*(\om B_1)=d*(\om B_2)=\frac{\om^{n-1}}{(n-1)!}dB_2.
$$
It follows that $dB_1=dB_2$. Then the second equation implies 
$$
n(B_1-B_2)=d^*d(B_2-B_1)=0.
$$
Hence, $B_1=B_2$.
\end{proof}

\begin{rmk}
We comment here on the generalization of the above functionals to those with zeroes being $\zeta$-flat connections. For any closed 2-form $\zeta\in\Om^2(M)$, an inner product structure on $\Om^*(M,Ad\,P)$ induces the decomposition of the curvature
$$
F=F_{\perp}+\Phi\zeta,
$$
where $F_{\perp}$ is orthogonal to $\zeta$. Thus, the Yang-Mills functional can also be decomposed as 
\begin{align}
\|F\|^2=\|F_{\perp}\|^2+\|\Phi\,\zeta\|^2,
\end{align}
and $\zeta$-flat connections are exactly the zero points of $\|F_{\perp}\|^2$ when $\dim M\geq 4$.

Similarly, if we set $d\theta=\zeta$ for the cone $\Om^*(M)[\theta]$, the functional $\|\tilde{F}\|_{\mathcal{C}}^2$ will become \begin{align}\|F_A+\zeta B\|^2+\|d_AB\|^2.
\end{align}
Its zero points $A+\theta B$ are then identifies with the $\zeta$-flat connections $A$.

However, unlike for a compatible symplectic structure $\om$, $*\,\zeta\neq\frac{\zeta^{n-1}}{(n-1)!}$ generally.  Hence, the properties described above do not generalize straightforwardly to the $\zeta$-flat case.
\end{rmk}



\section{Discussion: flows and characteristic classes}

In previous section, we introduced new Yang-Mills type functionals and explicitly showed that their Euler-Lagrange equations which specify the critical points can have different solutions.  It would be interesting to study the moduli space of solutions modulo gauge equivalence and identify any structures.  

It is also natural to consider how the critical points can be connected by gradient flows of the functionals.  Analogous to the Yang-Mills gradient flow, first studied in two dimension by Atiyah and Bott \cite{Atiyah-Bott}, we can write down a gradient flow for each functional that we have introduced.  As an example, take the primitive Yang-Mills functional of Definition \ref{dpYM}.   It gives what we can call the primitive Yang-Mills flow:
\begin{align}\label{pYMflow}
\dfrac{\partial A}{\partial t} = - d_A^* F_p\,.    
\end{align}
We expect this flow to share similar properties of the Yang-Mills flow.  The flow can be used to link critical points of the primitive Yang-Mills functional.  But additionally, we can also consider the flow of Yang-Mills solutions that are not primitive Yang-Mills.  The primitive Yang-Mills flow   \eqref{pYMflow} could therefore provide a link between critical points of the Yang-Mills functional and that of the primitive Yang-Mills.

Our motivation in writing down the new functionals in Section 3 was that the zeroes of the functionals correspond to symplectically-flat or more generally $\zeta$-flat connections.  If we remove this criterion, there are other functionals that are natural to write down.  A concrete example comes from the Chern-Simons characteristic forms.  We will describe how we can obtain such a functional for symplectically flat connections.  

Recall that a symplectically flat connection $A$ can be associated to a cone-flat connection $A+\theta\, B$ where $d\theta=\om$ and $F(A) = - B\,\omega$.  When $\om$ is integral class, cone-flat connections are just flat connections on the prequantum circle bundle whose Euler class is $\omega$ and $\theta$ is a global angular form.  We consider the Chern-Simons form of the Chern character, which we will denote by $Q_{2k+1}$.  Using the cone-flat connection, we decompose $Q_{2k+1}(A+\theta B)$ as $Q'_{2k+1}(A,B)+\theta Q''_{2k}(A,B)$ with both $Q'_{2k+1}(A,B), Q''_{2k}(A,B)\in\Omega^*(M,Ad\,P)$.  But note that if $A+\theta B$ is a cone-flat connection, then $d\,Q_{2k+1}=0$ and also $d\,Q''_{2k}(A,B)=0$.  Hence, $Q''_{2k}(A,B)$ is a characteristic form on $M$.  Furthermore, with $\dim M=2n$, $P_{2n}=\int_M Q''_{2n}(A,B)$ gives us a functional.  We will write down explicitly two cases, $\dim M=2, 4$.

For $\dim M=2$, we find
\begin{align}\label{bf2}
P_2(A,B)&=\frac{-1}{8\pi^2} \int_M \text{tr}\left[\omega\,B^2+2\,B\, F-d(A\,B)\right],
\end{align}
and for $\dim M=4$, we find
\begin{align}\label{bf4}
P_4(A,\Phi)=\frac{-i}{48} \int_M \text{tr}\left[3\,B\, F^2+3\,\omega\,B^2\,F+\omega^2\,B^3 - d(B dA\, A + \omega\, B\, A\, B +\frac{3}{2}B A^3 - B AdA)\right].
\end{align}
We will assume $M$ is a closed manifold.  Then the critical points of $P_2(A,B)$ are exactly symplectically flat connections, i.e. satisfying $F(A)=-B\,\omega$ and $d_A B =0$.   For $P_4(A,B)$, the symplectically flat connections are a subset of the critical points.  

Furthermore, if we substitute $F=\Phi\, \om$ and $B=-\Phi$ into \eqref{bf2} and \eqref{bf4}, we obtain the Chern-Simons invariants for symplectically flat connections
\begin{align*}
P_2(A, -\Phi) &=\frac{1}{8\pi^2} \int_{M^2} \text{tr}\,\Phi^2 \, \om \\
P_4(A, -\Phi) &=\frac{i}{48} \int_{M^4} \text{tr}\, \Phi^3\, \om^2
\end{align*}
which are just proportional to the integral of $\tr \,\Phi^{n+1}$ and generally are non-zero.

Finally, let us describe some characteristic classes of symplectically flat $G$-bundles, or more generally $\zeta$-flat $G$-bundles, when the Lie group $G$ is non-compact.  In essence, we can extend the theory of characteristic classes of flat bundles to $\zeta$-flat bundles by means of cone-flat connections.  (Note that a cone-flat connection is equivalent to a $\zeta$-flat connection if $d\theta=\zeta$, and to a symplectically flat connection in the special case where $d\theta=\om$.)  Our description will follow Morita's exposition of characteristic classes in \cite{Morita}*{Chapter 2.3}

Let $G$ be a Lie group and $\pi:P\to M$ be a  $\zeta$-flat $G$-bundle. Suppose $A\in\Omega^1(P,\mathfrak{g})$ is the corresponding connection form and its curvature is $\Phi\,\zeta$. Then $A$ induces a linear map
$$
\mathcal{A}:\mathfrak{g}^*\to\Omega^1(P)\oplus\theta\,\Omega^0(P),\qquad \alpha\mapsto\alpha\circ A -\theta\,\alpha(\Phi).
$$

$\mathcal{A}$ can be extended to a morphism from $\Lambda^*\mathfrak{g}^*$ to the cone $\Omega^*(P)[\theta]$ naturally. Explicitly, it sends $\alpha\in\Lambda^k\mathfrak{g}^*$ to $\alpha\circ A^{\otimes k}-\theta(\iota_{\Phi}\alpha\circ A^{\otimes (k-1)})\in \Omega^k(P)\oplus\theta\,\Omega^{k-1}(P)$. It is straightforward to verify that this morphism is a DGA morphism, where the differential on $\Lambda^*\mathfrak{g}^*$ is the exterior differentiation, and $d\theta=\zeta$ on $\Omega^*(P)[\theta]$. So it induces a morphism $H^*(\mathfrak{g})=H^*(\Lambda^*\mathfrak{g}^*)\to H^*(\Omega^*(P)[\theta])$ between cohomologies.

Take a maximal compact subgroup $K$ of $G$. Then the projection $\pi:P\to M$ can be decomposed as two maps
$$
P\to P/K\to M.
$$

Set $C^*(\mathfrak{g})=\Lambda^*\mathfrak{g}^*$ and $C^*(\mathfrak{g},K)$ be the space of $K$-invariant elements in $C^*(\mathfrak{g})$. That is,
\begin{align*}
C^k(\mathfrak{g},K) &=\{ \alpha\in C^k(\mathfrak{g})| \iota_X\alpha=0, \alpha(Ad_g X_1,\ldots,Ad_g X_k)=\alpha(X_1,\ldots,X_k) \\
& \quad \text{ for any } X,X_1,\ldots,X_k\in\mathfrak{k},g\in G  \}
\end{align*}
where $\mathfrak{k}$ is the Lie algebra of $K$. The image of $\mathcal{A}$ acting on $C^*(\mathfrak{g},K)$ is in the pullback of the projection from $\Omega^*(P/K)[\theta]$. So it induces the following commutative diagram
$$
\xymatrix{
C^*(\mathfrak{g}) \ar[r] & \Omega^*(P)[\theta] \\
C^*(\mathfrak{g},K) \ar[u] \ar[r] & \Omega^*(P/K)[\theta] \ar[u]
}
$$
So we obtain a morphism from $H^*(\mathfrak{g},K)$ to $H^*(\Omega^*(P/K)[\theta])$.

On the other hand, $P/K$ is a $G/K$ bundle over $M$ and it is well know that $G/K$ is diffeomorphic to a Euclidean space and in particular is contractible. Therefore, $P/K$ is homotopy equivalent to $M$ and we have $H^*(\Omega^*(P/K)[\theta])\simeq H^*(\Omega^*(M)[\theta])$.

The homomorphism $\mathcal{A}:H^*(\mathfrak{g},K)\to  H^*(\Omega^*(M)[\theta])$ is  the characteristic homomorphism, and for any $[\alpha]\in H^*(\mathfrak{g},K)$, the cohomology class $\mathcal{A}[\alpha]\in H^*(\Omega^*(M)[\theta])$ gives us the characteristic class of the $\zeta$-flat $G$-bundle corresponding to $[\alpha]$.  Notice that if $G$ is compact, then $K=G$ and the characteristic class becomes trivial since $H^*(\mathfrak{g},K)=0$.




\vskip 1cm
\noindent
{Department of Mathematics, University of California, Irvine, CA 92697, USA}\\
{\it Email address:}~{\tt lstseng@math.uci.edu}
\vskip .5 cm
\noindent
{Yau Mathematical Sciences Center, Tsinghua University, Beijing, 100084, China}\\
{\it Email address:}~{\tt jiaweiz1990@mail.tsinghua.edu.cn}

\end{document}